\documentclass[11pt]{amsart}
 \usepackage[centering]{geometry}
\usepackage{hyperref}
\usepackage{doi}
\usepackage{graphicx}
\usepackage{amssymb}
\usepackage{tikz-cd}
\usepackage{amsfonts,amsmath,amsthm,amscd,latexsym,euscript}
\usepackage{comment}

\usepackage[all]{xy}

\usepackage{epstopdf}
\title{Admissible replacements for simplicial monoidal model categories}
\author{Haldun  \"Ozg\"ur Bayindir}
\author{Boris Chorny} 

\newcommand{\op}{{\ensuremath{\textup{op}}}}

\newcommand {\cofib} {\ensuremath{\hookrightarrow}}

\newcommand {\trivcofib} {\ensuremath{\tilde\hookrightarrow}}

\newcommand {\otday} {\ensuremath{\otimes_{\text{Day}}}}

\newcommand {\trivfibr} {\ensuremath{\tilde\twoheadrightarrow}}

\newcommand {\we} {\ensuremath{\tilde\rightarrow}}

\DeclareMathOperator{\ind}{\textup{ind}}

\newcommand{\cal}[1]{\ensuremath{\mathcal #1}}

\newtheorem {theorem1}{Theorem}[subsection]
\newtheorem {theorem}[theorem1]{Theorem}

\newtheorem {corollary}[theorem1]{Corollary}
\newtheorem {proposition}[theorem1]{Proposition}
\newtheorem {lemma}[theorem1]{Lemma}

\theoremstyle{definition}
\newtheorem {definition}[theorem1]{Definition}
\newtheorem {notation}[theorem1]{Notation}

\newtheorem {construction}[theorem1]{Construction}
\theoremstyle{remark}
\newtheorem {remark}[theorem1]{Remark}

\newcommand{\cat}[1]{\ensuremath{\EuScript #1}}

\newcommand{\sS}{\ensuremath{\mathcal{S}} }

\DeclareMathOperator{\Hom}{\ensuremath{\textup{Hom}}}

\DeclareMathOperator{\map}{\textup{map}}

\newcommand{\II}{\ensuremath{\mathbb{I}}}

\newcommand{\colim}{\ensuremath{\mathop{\textup{colim}}}}
\newcommand{\hocolim}{\ensuremath{\mathop{\textup{hocolim}}}}

\newcommand{\I}{\mathbb{I}}

\newcommand{\Id}{\ensuremath{\textup{Id}}}

\newcommand{\Rosicky}{Rosick\'y }

\newcommand{\co}{\colon\thinspace}
\renewcommand{\hom}{\ensuremath{{\rm hom}}}

\newcounter{zahl}%
    {\end{list}}%

\begin{document}
\SelectTips{cm}{10}

\begin{abstract}
Using Dugger's construction of universal model categories, we produce replacements for simplicial and combinatorial symmetric monoidal model categories with better operadic properties. Namely, these replacements admit a model structure on  algebras over any given colored operad.

As an application, we show that in the stable case, such symmetric monoidal model categories are classified by commutative ring spectra when the monoidal unit is a compact generator. In other words, they are strong monoidally Quillen equivalent to modules over a uniquely determined commutative ring spectrum.


\end{abstract}
\maketitle

\section{Introduction}

A symmetric monoidal model category is a setting to study structured objects such as monoids, commutative monoids or modules using homotopy theoretic tools. However, one does not always have a model structure on these structured objects. For example, commutative monoids in chain complexes (commutative DGAs) is not known to carry a model structure induced from the underlying model category of chain complexes. The general problem of transferring a model structure to the categories of monoids and modules is studied by Schwede and Shipley in \cite{schwede2000algebrasandmodules} and it is shown that these lifts exist under mild hypothesis. In \cite[Section 4.5.4]{lurie2012higher}, Lurie studies this transfer problem for commutative monoids and this lifting problem requires a stronger hypothesis, which may be verified, though, for symmetric spectra with the positive stable model structure. More generally, one considers the algebras over a colored symmetric operad in a given symmetric monoidal model category. The transfer problem in this generality was studied  by Pavlov and Scholbach, \cite{Pavlov2018admissibility}.  

The main result of this paper is that any combinatorial simplicial  symmetric monoidal model category may be replaced, up to strong symmetric monoidal Quillen equivalence, by a model category allowing for algebras over any symmetric operad to have a model structure  transferred from the underlying category. As an application, we show that, under mild conditions, any stable symmetric monoidal model category is strong symmetric monoidally Quillen equivalent to a category of modules over \emph{commutative} ring spectrum. This result generalizes the Theorem by Schwede and Shipley, \cite[3.1.1]{schwede2003stablearemodules}.  In the framework of stable $\infty$-categories a similar result was obtained by Lurie, \cite[Proposition 7.1.2.7]{lurie2012higher}. Our approach provides a refinement of Lurie's theorem to the realm of stable combinatorial model categories.

Pavlov and Scholbach, \cite{pavlov2019symmetricopSsp}, call a symmetric monoidal model category \textbf{admissible} if for every colored symmetric operad $\cat O$, the category of  $\cat O$-algebras carry a model structure transferred from the underlying category. 
According to Pavlov and Scholbach, a monoidal model category is \textbf{nice} if it is h-monoidal, pretty small, flat and tractable. Detailed definitions may be found in Section \ref{sec Lpresheaf is nice}. 

 Pavlov and Scholbach's main result in \cite{pavlov2019symmetricopSsp} is that for a nice symmetric monoidal model category $\cat M$, there exists an admissible replacement of $\cat M$, up to a strong symmetric monoidal Quillen equivalence. This replacement is  the category of symmetric spectra over $\cat M$.


In this paper, we show that any  simplicial combinatorial symmetric monoidal model category is strong symmetric monoidally Quillen equivalent to a nice symmetric monoidal model category. Combining with the result of Pavlov and Scholbach, we obtain that  such model categories have replacements with admissible symmetric monoidal model categories. 

These replacements have many applications. For example, our admissible replacement is a Goerss-Hopkins context. This is again due to the results of Pavlov and Scholbach   \cite[Theorem 1.6]{pavlov2019symmetricopSsp}. This means that Goerss-Hopkins obstruction theory can be applied in this replacement to obtain commutative monoids from commutative monoids in the homotopy category. Furthermore, one has a strictification for $E_\infty$-algebras. In other words, $E_\infty$-algebras are Quillen equivalent to commutative monoids in this setting.

\begin{notation}
All our monoidal categories are symmetric monoidal. Therefore we say monoidal, when we mean symmetric monoidal. Similarly for model categories, functors and Quillen pairs.
\end{notation}
\begin{notation}
There are two notions of monoidal Quillen equivalences defined by Schwede and Shipley \cite{Schwede2003monoidalquillen}, weak monoidal and strong monoidal. All the monoidal Quillen equivalences we talk about are strong monoidal Quillen equivalences.
\end{notation}

\begin{theorem}\label{thm admisslbe rplcmnt}
Every combinatorial and simplicial symmetric monoidal model category is symmetric monoidally Quillen equivalent to an admissible  symmetric monoidal model category. Analogous result holds for monoidal model categories that are not necessarily symmetric.
\end{theorem}

\begin{remark}
In \cite{shipley2004convenient}, Shipley constructs a model structure for commutative ring spectra where cofibrant commutative ring spectra forget to  spectra that are cofibrant in a non-positive model structure. The admissible replacement provided by the proof the theorem above also satisfies a similar property. In this admissible replacement, cofibrant commutative monoids forget to cofibrant objects in a non-positive model structure. Indeed, this is true for algebras over various operads. This is a consequence of Pavlov-Scholbach \cite[Proposition 2.3.10 and Theorem 4.4]{pavlov2019symmetricopSsp}.
\end{remark}

The admissible replacement we construct satisfies further properties due to Pavlov-Scholbach \cite[Theorems 4.6 and 4.9]{pavlov2019symmetricopSsp}. These are stated in the following theorems.

\begin{theorem}\label{thm weak equivalences of operads}
Every combinatorial and simplicial symmetric monoidal model category is symmetric monoidally Quillen equivalent to an admissible symmetric monoidal model category $\cat C$ where a weak equivalence of operads $O \to P$ in \cat C induces a Quillen equivalence between the model categories of $O$-algebras and $P$-algebras.
\end{theorem}

\begin{remark}
We prove Theorems \ref{thm admisslbe rplcmnt} and \ref{thm weak equivalences of operads} in a more general setting, see Theorem \ref{thm generalized admssible and w e of operads}. Namely, these theorems are true for combinatorial \cat V symmetric monoidal model categories for every symmetric monoidal model category \cat V satisfying the axioms listed in Notation \ref{notation axioms for V}. For instance, the model categories of pointed simplicial sets and chain complexes over a field satisfy these axioms. 
\end{remark}

In the theorem below, \cat O denotes an operad in simplicial sets and  $\text{N}^{\otimes} \cat O$ denotes the operadic nerve of $\cat O$.
\begin{theorem}\label{thm infty category of algebras}
Every combinatorial and simplicial symmetric monoidal model category is symmetric monoidally Quillen equivalent to an admissible symmetric monoidal category \cat C where the underlying $\infty$-category of the model category of $\cat O$-algebras in \cat C is equivalent to the $\infty$-category of $\text{N}^{\otimes} \cat O$-algebras on the underlying $\infty$-category of \cat C.
\end{theorem}

To prove these results, we first replace the given monoidal model category by another one with a cofibrant monoidal unit if necessary. This is obtained using the results of Muro \cite{Muro2015unit}, see Section \ref{sec Muro cofibrant unit}. The replacement also satisfies the hypothesis of the theorems above. 

For a monoidal model category $\cat M$ as in the Theorem \ref{thm admisslbe rplcmnt} with cofibrant unit, we obtain the following zig-zag of Quillen equivalences.
\begin{equation} \label{diag zigzag to admissible}
 \begin{tikzcd}
& \cat M \arrow[r,swap,"E", shift right]
 & \arrow [l,swap,"D", shift right] \arrow[r,"F_0", shift left]	L_S[\cat M_{\lambda,cof}^{\op},\sS] & \arrow [l,"Ev_0", shift left]  Sp(L_S[\cat M_{\lambda,cof}^{\op},\sS],E)
  & 
 \end{tikzcd}
 \end{equation}

Here, $L_S[\cat M_{\lambda,cof}^{\op},\sS]$ denotes the universal model category construction of Dugger. This is a localization of the simplicial presheaves on a set of $\lambda$-presentable cofibrant objects of $\cat M$. We equip this category with a monoidal structure using the Day convolution. The arrows at the top denote the left adjoints. The right adjoint functor $E$ is the restricted Yoneda embedding and $D$ is the left adjoint to $E$ defined in \eqref{eq def of D}.

We show that $L_S[\cat M_{\lambda,cof}^{\op},\sS]$ satisfies further properties with respect to its monoidal structure. This allows us to use the results of Pavlov and Scholbach \cite{pavlov2019symmetricopSsp} to show that the monoidal model category $ Sp(L_S[\cat M_{\lambda,cof}^{\op},\sS],E)$ is admissible. To be precise, let $E$ denote the commutative monoid in symmetric sequences in $L_S[\cat M_{\lambda,cof}^{\op},\sS]$ given by the monoidal unit of $L_S[\cat M_{\lambda,cof}^{\op},\sS]$ at each degree. Here, $ Sp(L_S[\cat M_{\lambda,cof}^{\op},\sS],E)$ denotes the monoidal model category of  $E$-modules. Indeed, this is the same as the category of $I$-spaces in  $L_S[\cat M_{\lambda,cof}^{\op},\sS]$ of Sagave and Schlichtkrull \cite{sagave2012diagramspacessymspectra}. The adjoint pair on the right hand side is the standard one given in \cite[7.3]{hovey2001symspectrageneral}; for instance, $Ev_0$ is the evaluation at degree $0$. The pair $F_0 \dashv Ev_0$ is a Quillen pair when $Sp(L_S[\cat M_{\lambda,cof}^{\op},\sS],E)$ is given the stable model structure. Furthermore, $Sp(L_S[\cat M_{\lambda,cof}^{\op},\sS],E)$ is admissible with the positive stable model structure which is Quillen equivalent via the identity functor to the stable model structure. We omit this detail in the diagram above.


Gabriel's theorem provides a classification of co-complete abelian categories with a single small
projective generator as categories of modules. Schwede and Shipley prove a similar result for stable homotopy theory. They show that every stable, simplicial, cofibrantly generated and proper  model category with a compact generator  \cite[Theorem 3.1.1]{schwede2003stablearemodules} is Quillen equivalent to modules over a ring spectrum. We use Theorem \ref{thm admisslbe rplcmnt} to prove a monoidal  version of this result. 

\begin{theorem} \label{thm morita classification}
Let $\cat M$ be a stable, combinatorial and simplicial symmetric monoidal model category whose monoidal unit is a compact generator. In this situation, $\cat M$ is strong symmetric  monoidally Quillen equivalent to $R$-modules where $R$ is a commutative ring spectrum. 

Furthermore, $R$ is uniquely determined in the following sense. If the monoidal unit of $\cat M$ is cofibrant and if $\cat M$ is (strong or weak) symmetric monoidally Quillen equivalent to $R'$-modules for another commutative ring spectrum $R'$ where each monoidal model category in the zig-zag has a cofibrant unit, then $R$ and $R'$ are weakly equivalent as commutative ring spectra.

\end{theorem}

The proof of this result makes essential use of the proof of Theorem \ref{thm admisslbe rplcmnt}. Firstly, in the zig-zag of Quillen equivalences in \eqref{diag zigzag to admissible}, instead of  $E$, we use the symmetric sequence corresponding to the simplicial tensor $S^1 \otimes -$. This provides a spectral replacement of the given model category which allows us to obtain the derived endomorphism spectrum of the unit. To make sure this derived endomorphism spectrum is a  commutative ring spectrum, one needs to take a fibrant replacement of the unit as a commutative ring spectrum. For this, we make use of admissibility.

As we mentioned earlier, the $\infty$-categorical analogue of this result is due to Lurie \cite[Proposition  7.1.2.7]{lurie2012higher}. To our knowledge, there is no straightforward way of obtaining our result from Lurie's. Nikolaus and Sagave show that presentably symmetric monoidal $\infty$-categories come from symmetric monoidal model categories \cite{nikolaus2017presentably}. However, if the given presentably symmetric monoidal $\infty$-category is obtained from a symmetric monoidal model category, it is not known if the construction of Nikolaus and Sagave  gives back the monoidal model category one starts with. On the other hand, the uniqueness part of Theorem \ref{thm morita classification} follows by Lurie's theorem.

\textbf{Outline:} In Section \ref{sec preliminaries}, we start with a discussion on monoidal model categories and monoidal Quillen equivalences. After that, we define Day convolution and combinatorial monoidal model categories. Section \ref{sec admissible replacement} is devoted to the proof of Theorems \ref{thm admisslbe rplcmnt}, \ref{thm weak equivalences of operads} and \ref{thm infty category of algebras}. In Section \ref{sec categories of modules}, we prove Theorem \ref{thm morita classification}.

\textbf{Acknowledgenements:} We would like to thank the referee for his or her helpful remarks. The research of the second author was partially supported by ISF grant 1138/16. The first author was also supported by the same grant. Furthermore, the first author would like to thank the University of Haifa for the hospitality and support he received while he was working on this project.

\section{Preliminaries} \label{sec preliminaries}

\subsection{Monoidal model categories} \label{sec prelim monoidal model categories}

We recall the theory of monoidal model categories and enriched monoidal model categories. As mentioned earlier, we say monoidal when we mean symmetric monoidal for categories, model categories and functors. Furthermore, by a monoidal category, we mean a closed symmetric monoidal category.  See Definitions 4.1.1, 4.1.4 and 4.1.12 in Hovey \cite{Hovey} for the definition of a closed symmetric monoidal categories. In the following definition, our unit axiom is stronger than that of \cite{Hovey} since we don't assume $X$ to be cofibrant. This is what is called the very strong unit axiom by Muro in \cite{Muro2015unit}.

\begin{definition}
A \emph{monoidal model category} $\cat M$  is a model category whose underlying category is a monoidal category $(\cat M,\otimes,\I)$ with a product $\otimes$ and a unit $\I$ such that the monoidal structure satisfies the following compatibility conditions with respect to the model structure on $\cat M$.

\begin{enumerate}
    \item \textbf{Pushout-product axiom:} For two cofibrations $f \co U \to V$ and $g \co X \to Y$, the following map
    \[f\square g \co U \otimes Y \amalg_{U \otimes X} V \otimes X \to V \otimes Y\]
    is a cofibration. Furthermore, this is a weak equivalence if either $f$ or $g$ is a weak equivalence.
    \item \textbf{Unit axiom:} There is a cofibrant replacement $c \I\we \I$  of the unit such that for every $X$ in $\cat M$, the following map 
    \[(c \I)\otimes X \to \I \otimes X \cong X\]
    is a weak equivalence.
\end{enumerate}

\end{definition}
\begin{remark}
The second axiom above is satisfied if $\I$ is cofibrant.
\end{remark}

\begin{remark}
Schwede and Shipley define weak and strong monoidal Quillen equivalences in \cite{Schwede2003monoidalquillen}. They show that monoidal Quillen equivalences induce Quillen equivalences at the level of monoids and modules, see \cite[Theorem 3.12]{Schwede2003monoidalquillen}. All our Quillen equivalences are indeed strong monoidal Quillen equivalences, therefore, when we say monoidal Quillen equivalence, we mean strong monoidal Quillen equivalence. Except  Section \ref{sec Muro cofibrant unit},  we only consider monoidal Quillen equivalences between monoidal model categories with cofibrant units.
\end{remark} 

A \textbf{monoidal functor} $F \co \cat C \to \cat D$ between two monoidal categories $(\cat C,\otimes_{\cat C},\I_{\cat C})$ and $(\cat D,\otimes_{\cat D},\I_{\cat D})$ is a functor that is equipped with natural isomorphisms 
\[F(C_1) \otimes_{\cat D} F(C_2) \cong F(C_1 \otimes_{\cat C} C_2)\]
and 
\[I_{\cat D}\cong F(I_{\cat C})\]
which are coherently symmetric, associative  and unital, see Borceux  \cite[6.4.1]{borceux1994handbook2}.

\begin{definition} \label{def monoidal quillen equivalence} (Schwede-Shipley \cite[Definition  3.6]{Schwede2003monoidalquillen}) 
A \emph{monoidal Quillen equivalence} between monoidal model categories with cofibrant  units is a Quillen equivalence where the left adjoint is a monoidal functor.
\end{definition}

For a symmetric monoidal model category $\cat V$, a $\cat V$ model category is a model category with an action of $\cat V$ that is compatible with the model structures, see Hovey \cite[Definition  4.2.18]{Hovey}. A $\cat V$ monoidal model category is a $\cat V$ model category where the $\cat V$ action is compatible with the monoidal structure. This can be formulated as in the following definition.
\begin{definition} \label{def V monoidal model category} (Hovey \cite[Definition 4.2.20]{Hovey})
Let $\cat V$ and $\cat C$ be monoidal model categories. We say that $\cat C$ is a $V$ monoidal model category if there is a left Quillen monoidal functor 
\[F \co \cat V \to \cat C.\]
\end{definition}
In this situation, the action of an object $V$ in $\cat V$ on $C$ in $\cat C$ is given by 
\[V \otimes C := F(V) \otimes_{\cat C} C.\]

\subsection{Day convolution for monoidal categories} \label{sec Day conv}
Let \cat V be a monoidal category. Given a small monoidal \cat V enriched category $(\cat C,\otimes,\I)$, we consider the category of $V$-enriched functors $[\cat C^{op},\cat V]$. We use the following monoidal structure on $[\cat C^{op},\cat V]$ due to Day \cite{day1970dayconvolution}. 
\[
\forall F,G\in \cat V^{\cat C^{\op}}, \quad F\otimes_{\text{Day}}G= \int^{C_1,C_2\in \cat C}\hom_{\cat C}(-,C_1\otimes C_2)\otimes F(C_1)\otimes G(C_2)
\]
This is called the Day convolution. With this monoidal product, $([\cat C^{op},\cat V],\otimes_{Day},Y(\I))$ becomes a closed symmetric monoidal category where 
\[Y \co \cat C \to [\cat C^{op}, \cat V]\]
denotes the Yoneda embedding given by 
\[Y(C)(-) = \hom_{\cat C}(-,C)\]
for every $C$ in \cat C. We  use Day convolution as our monoidal product on the presheaf category because  $Y$ becomes a strong monoidal functor in this situation. 

There is also the point-wise monoidal structure on the presheaf category $ [\cat C^{op}, \cat V]$ but this does not suit our purposes. This is because this monoidal product makes no reference to the monoidal structure on $\cat C$ and therefore  does not render $Y$ into a strong monoidal functor in general. 
\subsection{Combinatorial monoidal model categories}
Here, we provide an overview of combinatorial model categories and we define what we mean by combinatorial monoidal model categories.

Let $\cat C$ be a category and let $C$ be an object in $\cat C$. For a regular cardinal $\lambda$, we say $C$ is \textbf{$\lambda$-presentable} if mapping out of
$C$ commutes with $\lambda$-filtered colimits.


A category $\cat C$ is said to be \textbf{locally $\lambda$-presentable} if it is co-complete and if there is a set of $\lambda$-presentable objects in $\cat C$ such that every object of $\cat C$ is a $\lambda$-filtered colimit of objects in this set. We say $\cat C$ is \textbf{locally presentable} if it is \textbf{locally $\lambda$-presentable} for some regular cardinal $\lambda$. 

We say a model category $\cat M$ is \textbf{$\lambda$-combinatorial} if it admits sets of generating (acyclic) cofibrations with $\lambda$-presentable (co)domains and if it is locally $\lambda$-presentable, Barwick \cite[Definition 1.21]{Barwick-loc}. We say  $\cat M$ is  \textbf{combinatorial}   if it is  $\lambda$-combinatorial  for some regular cardinal $\lambda$.

In a $\lambda$-locally presentable category $\cat C$, isomorphism classes of $\lambda$-presentable objects form a set \cite[Section 2]{Dugger-generation}, we call the corresponding small full subcategory $\cat C_{\lambda}$. For a $\lambda$-combinatorial model category $\cat M$, we denote the cofibrant objects of $\cat M_{\lambda}$  by $\cat M_{\lambda,cof}$.

For monoidal model categories, we use a stronger notion of combinatoriality. For this, we use the definition of locally $\lambda$-presentable base due to  Borceux, Quinteiro and \Rosicky  \cite[Definition 1.1]{enriched-sketches}. A monoidal category $C$ is said to be  a \textbf{locally  $\lambda$-presentable base} if it is a  locally $\lambda$-presentable category and if  $\cat C_{\lambda}$ contains the monoidal unit and if $\cat C_{\lambda}$ is closed under monoidal products.

\begin{definition} \label{def combinatorial monoidal}
A monoidal model category $\cat M$ is  $\lambda$-combinatorial if it is cofibrantly generated and its underlying monoidal category is a  locally  $\lambda$-presentable base. In other words, if $\cat M$ is $\lambda$-combinatorial as a model category and if the monoidal structure on $\cat M$ gives a monoidal structure on $\cat M_{\lambda}$, then we say the monoidal model category $\cat M$ is $\lambda$-combinatorial. This is equivalent to $\cat M$ being $\lambda$-combinatorial as a category enriched over itself.

Furthermore, $\cat M$ is said to be combinatorial if it is $\lambda$-combinatorial for some regular cardinal $\lambda$.
\end{definition}

\begin{lemma}\label{lem homotopycolimits}
Let \cat N be a model category and $J$ be a regular cardinal. Assume that \cat N admits a set of generating cofibrations with $J$-presentable (co)domains. Given a $J$-filtered colimit $\colim_{J'} N_j$ for some $J \leq J'$, the canonical map
\[\hocolim_{J'} N_j \we \colim_{J'} N_j\]
is a weak equivalence.
\end{lemma}
\begin{proof}
There is an adjunction  \begin{equation*}
 \begin{tikzcd}
 & \cat N \arrow[r,swap,"cnst", shift right] 
 & \arrow [l,swap,"colim", shift right] [J',\cat N]
 \end{tikzcd}
 \end{equation*}
where the right adjoint  $cnst$ provides the constant diagram and the left adjoint denoted by $colim$ takes the corresponding colimit in $\cat N$. Furthermore, the functor category is given the projective model structure. Let $N_*$ denote the corresponding diagram in the diagram category and let a trivial fibration $\tilde{N}_* \trivfibr N_*$ provide a cofibrant replacement of $N_*$. We need to show that the induced map  
\[colim (\tilde{N}_*) = \hocolim_{J'} N_j \to \colim_{J'} N_j\]
is a weak equivalence where the equality above follows by the definition of homotopy colimits. Indeed, we show that the map above is a weak equivalence by showing that it satisfies the left lifting property with respect to the generating cofibrations of $\cat N$. Let $A \cofib B$ denote a generating cofibration of $\cat N$. We need to solve the lifting problem 
\begin{equation*}
    \begin{tikzcd}
     A \ar[r] \ar[d,hook] & \colim_{J'} \tilde{N}_j \ar[d]\\
     B \ar[r] \ar[ur,dashed] &  \colim_{J'} N_j.
    \end{tikzcd}
\end{equation*}
Since $A$ and $B$ are $J$-presentable, we obtain that the the horizontal maps out of them in the diagram above factors through $\tilde{N}_j \to N_j$ for some $j$. Since $\tilde{N}_* \trivfibr N_*$ is a levelwise trivial fibration, the map $\tilde{N}_j \to N_j$  is a trivial fibration and therefore one obtains the desired lifting. 
\end{proof}

The following corollary should be compared to \cite[Proposition 7.3]{Dugger-generation}.
\begin{corollary}
Let \cat N be a model category and $J$ be a regular cardinal. Assume that \cat N admits a set of generating cofibrations with $J$-presentable (co)domains. Furthermore, let $J'$ be a regular cardinal with $J \leq J'$. Given a levelwise equivalence $N_*\we N_*'$ of diagrams over  $J'$, the induced map 
\[\colim_{J'} N_j \to \colim_{J'} N_j'\]
is a weak equivalence.
\end{corollary}

\begin{proof}
There is a commuting diagram 
\begin{equation*}
\begin{tikzcd}
    \ar[r,"\simeq"] \hocolim_{J'} N_j  \ar[d] & \hocolim_{J'} N_j' \ar[d] \\
    \colim_{J'} N_j \ar[r]& \colim_{J'} N_j'.
    \end{tikzcd}
\end{equation*}
Here, the vertical maps are weak equivalences due to Lemma \ref{lem homotopycolimits}. The result follows by the two out of three property of weak equivalences.
\end{proof}

\section{Admissible replacement} \label{sec admissible replacement}
In this section, we prove Theorem \ref{thm admisslbe rplcmnt}. In other words, we show that every  combinatorial and  simplicial monoidal model category is monoidally Quillen equivalent to an admissible model category.
More generally we prove this result for  combinatorial $\cat V$ monoidal model categories where $\cat V$ denotes a symmetric monoidal model category satisfying the axioms stated in the following. 
\begin{notation}\label{notation axioms for V}
For the rest of this work, let $\cat V$ denote a combinatorial symmetric monoidal model category satisfying the following properties. 
\begin{enumerate}
    \item Every object of \cat V is cofibrant. 
    \item There is a set of generating cofibrations of $\cat V$ where (co)domains of the generating cofibrations are $\aleph_0$-presentable. 
    \item The model category $\cat V$ is left proper.
 \end{enumerate}
\end{notation}

The axioms above are satisfied by the model categories of simplicial sets, pointed simplicial sets and chain complexes over a field $k$. 

For our constructions, it is important that we start with a monoidal model category whose monoidal unit is cofibrant. In Section \ref{sec Muro cofibrant unit}, we use a theorem of Muro \cite{Muro2015unit} to show that a given monoidal model category satisfying the hypothesis of Theorem \ref{thm admisslbe rplcmnt} can be replaced with a monoidal model category whose unit is cofibrant. Furthermore, this replacement is also a  combinatorial \cat V symmetric monoidal  model category. Therefore, for the rest of this section, we assume that we start with a monoidal model category with a cofibrant unit.

Let $\cat M$ be a monoidal model category as in Theorem \ref{thm admisslbe rplcmnt} whose monoidal unit is cofibrant. To prove Theorem \ref{thm admisslbe rplcmnt}, we need to construct the zig-zag of monoidal Quillen equivalences in \eqref{diag zigzag to admissible} and prove that $Sp(L_S[\cat M_{\lambda,cof}^{\op},\cat V])$ is admissible. The first Quillen equivalence is constructed in Section \ref{sec replacement with presheaf category}. Section \ref{sec Lpresheaf is nice} is devoted to the proof of the fact that   $L_S[\cat M_{\lambda,cof}^{\op},\cat V]$ is nice in the sense of Pavlov and Scholbach \cite[Definition 2.3.1]{pavlov2019symmetricopSsp}. Proposition \ref{prop spectral replacement} provides the Quillen equivalence on the right hand side and the admissibility of $Sp(L_S[\cat M_{\lambda,cof}^{\op},\cat V])$ is given in Theorem \ref{thm sp is admissble for every nice}.

\subsection{Cofibrant monoidal unit}
\label{sec Muro cofibrant unit}
 Using Muro's results, we show that every combinatorial and \cat V monoidal model category carries a monoidally Quillen equivalent  model structure where the monoidal unit is cofibrant \cite{Muro2015unit}. 

Let $\cat M$ be a combinatorial \cat V monoidal model category. Theorem 1 in Muro \cite{Muro2015unit} provides a new model structure $\tilde{\cat M}$ on the same underlying category whose monoidal unit is cofibrant. The weak equivalences of $\tilde{\cat M}$ and $\cat M$ are the same but $\tilde{\cat M}$ possibly has more cofibrations than $\cat M$, i.e.\ cofibrations of $\cat M$ are also cofibrations in $\tilde{\cat M}$. In particular, the identity functor is a left Quillen functor 
\[\cat M \to \tilde{\cat M}\]
which is the left adjoint of a monoidal Quillen equivalence. Since the unit of $\cat M$ is not cofibrant, we refer the reader to Schwede and Shipley's Definition 3.6 in \cite{Schwede2003monoidalquillen} for the definition of monoidal Quillen equivalences instead of the one given in Section \ref{sec prelim monoidal model categories}. 

Furthermore, $\tilde{\cat M}$ is combinatorial and monoidal, Muro \cite[Theorem 1]{Muro2015unit}. We only need to show that $\tilde{\cat M}$ is also a \cat V monoidal model category. This amounts to having a monoidal left Quillen functor $\cat V \to \tilde{\cat M}$. Since $\cat M$ is \cat V monoidal, there is a monoidal left Quillen functor 
\[F \co \cat V \to \cat M\]
and composing this with the left Quillen functor induced by the identity functor, we obtain the desired functor  $\cat V \to \tilde{\cat M}$. This shows that $\tilde{\cat M}$ is a \cat V monoidal model category. We obtain the following version of Muro's theorem.

\begin{theorem}[Muro \cite{Muro2015unit}] \label{thm Muro cofibrant unit}
Every combinatorial \cat V symmetric monoidal model category is symmetric monoidally Quillen equivalent to a combinatorial \cat V symmetric monoidal model category whose unit is cofibrant.
\end{theorem}

\subsection{Replacement with the presheaf category} \label{sec replacement with presheaf category}

Here, we construct the first monoidal Quillen equivalence in \eqref{diag zigzag to admissible}. This is analogous to Dugger's construction of universal model categories. We obtain a replacement with  a localization of the \cat V-enriched presheaves on the cofibrant $\lambda$-presentable objects of the given monoidal model category (for a sufficiently large cardinal $\lambda$). This presheaf category equipped with the Day convolution satisfies further properties on its monoidal structure and this guarantees that the symmetric spectra on the presheaf category (with the positive stable model structure) is admissible. 

Let $(\cat M,\wedge,\mathbb{I_{\cat M}})$ be a combinatorial \cat V monoidal model category with a cofibrant monoidal unit. Let $\lambda$ denote a regular cardinal  for which the following are satisfied.

\begin{enumerate}
    \item The symmetric monoidal model category $\cat M$ is  $\lambda$-combinatorial.
    \item The cofibrant replacement functor in $\cat M$ preserves $\lambda$-filtered colimits. 
    \item A cofibrant replacement of a $\lambda$-presentable object is $\lambda$-presentable.
\end{enumerate}
The last two items above follow by   Dugger \cite[Proposition 2.3]{Dugger-generation},

Suppose $\cat M$ is $\lambda$-combinatorial for some cardinal $\lambda$ and let  $\cat M_{\lambda,cof}$ denote the subcategory of $\lambda$-presentable  cofibrant objects. 

We consider the category of $\cat V$-enriched functors and $\cat V$-natural transformations from $\cat M_{\lambda,cof}^{op}$ to $\cat V$. Let $[\cat M_{\lambda,cof}^{\op},\cat V]$ denote this category. There is a fully faithful functor 
\[Y \co \cat M_{\lambda,cof} \to [\cat M_{\lambda,cof}^{\op},\cat V]\] given by the Yoneda embedding. In other words, $Y(M) = \hom(-,M)$ for every $M \in \cat M_{\lambda,cof}$. For the inclusion functor $I \co \cat M_{\lambda,cof} \to \cat M$, there is the following left Kan extension  
\begin{equation} \label{diag lKan}
 \begin{tikzcd}
 \cat M_{\lambda,cof} \arrow[dr,"Y"] \arrow[d,"I"] &  \\
 \cat M & \arrow[l,dashed,"D"] [\cat M_{\lambda,cof}^{\op},\cat V{]}
 \end{tikzcd}
 \end{equation}
 which makes the above diagram commute up to a natural isomorphism, see Kelly \cite[Proposition 4.23]{Kelly}. Furthermore by Kelly \cite[Equation 4.25]{Kelly} and the Yoneda lemma, $D$ is given by the following coend
 \begin{equation}  \label{eq def of D}
 D(F)=\int^{M\in \cat M_{\lambda,cof}}F(M)\otimes I(M).
 \end{equation}
 By Kelly \cite[Equations 3.5 and 3.70]{Kelly}, $D$ is the left adjoint of the restricted Yoneda functor $E$ given by $E(M)(M_i) = \hom(M_i,M)$ for each $M \in \cat M$ and $M_i \in \cat M_{\lambda,cof}^{op}$. We obtain the following adjoint pair:
 \begin{equation*} \label{diag the first pair}
 \begin{tikzcd}
 & \cat M \arrow[r,swap,"E", shift right]
 & \arrow [l,swap,"D", shift right] [\cat M_{\lambda,cof}^{\op},\cat V].
 \end{tikzcd}
 \end{equation*}

There is a model structure on $[\cat M_{\lambda,cof}^{\op},\cat V{]}$ where weak equivalences and fibrations are given by levelwise weak equivalences and levelwise fibrations. This is called the projective model structure. The generating cofibrations and trivial cofibrations of the projective model structure are given by
\begin{equation} \label{eq gencof for presheaves}
    \begin{split}
        I^\prime =& \{Y(A) \otimes i \mid A \in \cat M_{\lambda,cof}, i \in I\}, \\
        J^\prime =& \{Y(A) \otimes j \mid A \in \cat M_{\lambda,cof}, j \in J\}, 
    \end{split}
\end{equation}
respectively where $I$ and $J$ denote the canonical generating sets of cofibrations and trivial cofibrations of $\cat V$ respectively. By \eqref{eq def of D}, it is clear that $D$ commutes with tensoring with morphisms in $\cat V$. Therefore, we have 
\[D(Y(A) \otimes i) = D(Y(A)) \otimes i \cong A \otimes i\]
for every object $A$ in $\cat M_{\lambda,cof}$ and morphism $i$ in $\cat V$. Since $\cat M$ is a $\cat V$ model category and $A$ is cofibrant, $A\otimes i$ is a (trivial) cofibration for every (trivial) cofibration $i$.  This shows that $D$ preserves generating cofibrations and trivial  cofibrations. We obtain that $D \dashv E$ is indeed a Quillen pair. 

Now we prove that the left Quillen functor $D$ is homotopically surjective, see Dugger  \cite[Definition 3.1]{Dugger-generation}. We start with the following lemma.

\begin{lemma}\label{lemma presheaves is pretty small}
The generating cofibrations of $[\cat M_{\lambda,cof}^{\op},\cat V{]}$ given above have $\aleph_0$-presentable (co)domains.
\end{lemma}
\begin{proof}
We need to show that the morphisms in $I'$ have $\aleph_0$-presentable (co)domains. Due to our standing assumptions, there is a set of generating cofibrations $I$  of $\cat V$ consisting of maps with  $\aleph_0$-presentable (co)domains. In what follows, for a given category $\cat C$, we denote the set of morphisms in $\cat C$ by $\cat C (-,-)$.

Given a filtered colimit $\colim_{j \in J} M_j$ in $[\cat M_{\lambda,cof}^{\op},\cat V]$ for some  $\aleph_0$-filtered category $J$, an $\aleph_0$-presentable $K$ in $\cat V$ and an object $A$ in $\cat M_{\lambda,cof}$, we have the following equalities.
\begin{equation*}
    \begin{split}
       [\cat M_{\lambda,cof}^{\op},\cat V](R_A \otimes K, \colim_{j \in J} M_j) \cong& \cat V(K,\hom(R_A, \colim_{j \in J} M_j))\\
        \cong& \cat V(K,(\colim_{j \in J} M_j) (A))\\
        \cong & \cat V (K, \colim_{j \in J} \hom(R_A, M_j)) \\
        \cong & \colim_{j \in J} \cat V(K, \hom(R_A,M_j))  \\
        \cong & \colim_{j \in J} [\cat M_{\lambda,cof}^{\op},\cat V](R_A \otimes K,M_j)
    \end{split}
\end{equation*}
The second and the third equalities follow by the Yoneda lemma and the first and the last equalities follow because $[\cat M_{\lambda,cof}^{\op},\cat V]$ is a $\cat V$-model category. The fourth equality follows by the assumption that $K$ is $\aleph_0$-presentable. In the equalities above, $\hom(-,-)$ denotes the $\cat V$-enriched maps in  $[\cat M_{\lambda,cof}^{\op},\cat V]$. Since the (co)domains of the maps in $I'$ are of the form $R_A \otimes K$, the chain of isomorphisms above provides the desired result.
\end{proof}

\begin{proposition}
The left Quillen functor $D$ is homotopically surjective. 
\end{proposition}
\begin{proof}
We need to show that for every fibrant $M \in \cat M$, the natural map 
\begin{equation*}
D[E(M)]^{cof} \to M 
\end{equation*}
is a weak equivalence where $[E(M)]^{cof}$ denotes a cofibrant replacement of $E(M)$. There is a $J$-filtered colimit 
\[M = \colim_J M_j\]
for some $\lambda \leq J$ such that
each $M_j$ is in ${\cat M}_{\lambda}$. Using the functorial cofibrant replacement in $\cat M$ that preserves $\lambda$-filtered colimits, one obtains another $J$-diagram $\tilde{M}_*$ with compatible trivial fibrations $\tilde{M}_j \trivfibr M_j$ such that each $\tilde{M}_j$ is in $\cat M_{\lambda,cof}$. This factorization guarantees that  the induced map 
\[\colim_J \tilde{M}_j \trivfibr \colim_J M_j\]
is a trivial fibration. In particular, $ \colim_J \tilde{M}_j$ is fibrant. Let $\tilde{M}$ denote  $\colim_J \tilde{M}_j$.

Due to the following diagram 
\begin{equation*}
    \begin{tikzcd}
     D[E(\tilde{M})]^{cof} \ar[r]\ar[d,"\simeq"] &\tilde{M}\ar[d,"\simeq"]\\ 
    \ar[r]  D[E(M)]^{cof} & M, 
    \end{tikzcd}
\end{equation*}
it is sufficient to show that the top horizontal map above is a weak equivalence. The left hand vertical map above is a  weak equivalence as  $E$ (resp.\ $D$) preserves weak equivalences between fibrant (resp.\ cofibrant) objects.

We obtain a cofibrant replacement of $E(\tilde{M})$ as 
\[\hocolim_J \hom(-,\tilde{M_j}) \we \colim_J \hom(-,\tilde{M}_j) \cong E(\tilde M) \]
where the map above is a weak equivalence due to Lemmas \ref{lemma presheaves is pretty small} and \ref{lem homotopycolimits}. The  isomorphism above follows because
\[E(\tilde{M})(N) = \hom(N,\colim_J \tilde{M}_j) \] 
and because mapping out of the objects of $\cat M_{\lambda,cof}$ preserves $\lambda$-filtered colimits.

The top horizontal map in the  square above is given by composing in the following commuting diagram. 
\begin{equation*}
    \begin{tikzcd}[row sep=normal, column sep = tiny]
     D[E(\tilde{M})]^{cof} = D(\hocolim_J \hom(-,\tilde{M}_i)) \ar[r,"\cong"] \ar[d]&\hocolim_J D(\hom(-,\tilde{M}_j))\ar[d,"\simeq"]&\\
     D(\colim_J \hom(-,\tilde{M}_j)) \ar[r,"\cong"]& \colim_J D(\hom(-,\tilde{M}_j)) \ar[r,"\cong"]& \colim_J \tilde{M}_j = \tilde{M}
    \end{tikzcd}
\end{equation*}
Each $\hom(-,\tilde{M}_j)$ above is cofibrant, see \eqref{eq gencof for presheaves}.  Therefore,  $D(\hom(-,\tilde{M}_j))$ provides the correct homotopy type and one obtains that the top horizontal arrow is an equivalence as $D$ preserves homotopy colimits between cofibrant objects. The vertical arrow on the right hand side is a weak equivalence due to Lemma \ref{lem homotopycolimits}. Finally, the isomorphisms $D(\hom(-,\tilde{M}_j))\cong \tilde{M}_j$ follows by Diagram \eqref{diag lKan}. This provides the desired result.

\end{proof}

Furthermore, $D \dashv E$ becomes a Quillen equivalence after a localization of 
$[\cat M_{\lambda,cof}^{\op},\cat V]$. Since $[\cat M_{\lambda,cof}^{\op},\cat V]$ is left proper and combinatorial, this follows by Dugger  \cite[Proposition 3.2]{Dugger-generation}. Indeed, this localization can be described as follows. Let $\lambda'$ be a regular cardinal and let $S_{\lambda'}$ be the set of maps given by the first factor in the factorization of the natural maps 
\begin{equation} \label{eq localization maps}
    N \to E(fD(N))
\end{equation}
as a cofibration followed by a trivial fibration for every cofibrant and   $N$ in $[\cat M_{\lambda,cof}^{\op},\cat V]_{\lambda'}$ where $f$ denotes a fibrant replacement functor. As before, $[\cat M_{\lambda,cof}^{\op},\cat V]_{\lambda'}$ denotes the subcategory of $\lambda'$-presentable objects in $[\cat M_{\lambda,cof}^{\op},\cat V]$. The proof of Dugger \cite[Proposition 3.2]{Dugger-generation} shows that there is a large enough $\lambda'$ for which the localization with respect to $S_{\lambda'}$ renders $D \dashv E$ into a Quillen equivalence. Let $S$ denote  $S_{\lambda'}$ for a chosen large enough $\lambda'$. We obtain the following.  \begin{proposition}\label{prop replacement with the presheaf}
For the set of maps $S$ defined above, the Quillen adjoint pair $D \dashv E$  induces a Quillen equivalence:
 \begin{equation*} \label{diag the first pair after localization}
 \begin{tikzcd}
 & \cat M \arrow[r,swap,"E", shift right]
 & \arrow [l,swap,"D", shift right] L_S[\cat M_{\lambda,cof}^{\op},\cat V].
 \end{tikzcd}
 \end{equation*}
\end{proposition}

\subsection{The monoidal structure on the presheaf category}
We equip $[\cat M_{\lambda,cof}^{\op},\cat V]$ with a symmetric monoidal product using the  Day convolution, see Section \ref{sec Day conv}. This makes $[\cat M_{\lambda,cof}^{\op},\cat V]$ a monoidal model category, Batanin-Berger \cite[Theorem 4.1]{batanin2017algoverpolynomial}. 

The levelwise monoidal structure on $[\cat M_{\lambda,cof}^{\op},\cat V]$ is not suitable for our purposes because it makes no reference to the monoidal structure on $\cat M$ and therefore does not render $D\dashv E$ into a monoidal Quillen pair in general.

In order to use the Day convolution, we need to show that $\cat M_{\lambda,cof}^{op}$ is a  monoidal category. 

\begin{proposition}
The monoidal structure on $\cat M$ induces a monoidal structure on $\cat M_{\lambda,cof}^{op}$. 
\end{proposition}
\begin{proof}
It is sufficient to note that $\cat M_{\lambda,cof}$ is closed under the monoidal product and that the unit is in $\cat M_{\lambda,cof}$. This is true for $\cat M_{\lambda}$ because it is part of our definition of combinatorial monoidal model categories, see Definition \ref{def combinatorial monoidal}. The unit of the monoidal structure is also assumed to be cofibrant and due to the pushout product axiom, monoidal product of cofibrant objects is cofibrant. 
\end{proof}
Finally, we show that $D \dashv E$ is a monoidal Quillen equivalence. 
\begin{proposition} \label{prop strongmonoidalDE}
The Quillen equivalence $D \dashv E$ between
$\cat M$ and 	$L_S[\cat M_{\lambda,cof}^{\op},\cat V]$ is a strong  monoidal Quillen equivalence. 
\end{proposition}

\begin{proof}
It is sufficient to show that $D$ is a strong  symmetric monoidal functor, see Definition \ref{def monoidal quillen equivalence}. The unit map of the monoidal structure of $D$ is provided by Diagram \eqref{diag lKan}. 

For the monoidal structure map of $D$, we consider two functors $F_1$ and $F_2$ that solve the extension problem in the following diagram. 

\begin{equation} \label{diag lKanformonoidality}
 \begin{tikzcd}
 \cat M_{\lambda,cof} \times \cat M_{\lambda,cof} \arrow[dr,"Y \times Y"] \arrow[d,"\wedge \circ (I\times I)"] &  \\
 \cat M & \arrow[l,dashed,"F_i"] [\cat M_{\lambda,cof}^{\op},\cat V{] \times [\cat M_{\lambda,cof}^{\op},\cat V{]}}
 \end{tikzcd}
 \end{equation}
 Here, $\wedge$ denotes the smash product functor $\cat M \times \cat M \to \cat M$ of $\cat M$.  Furthermore, the functors $F_i$ are defined by the formulas  
\[F_1(A,B) = D(A\otday B)\]
and 
\[F_2(A,B) = D(A) \wedge D(B).\]
Using the fact that $Y$ is strong monoidal and the natural isomorphism that makes Diagram \eqref{diag lKan} commute, one observes that for each $i$, there is a canonical natural isomorphism that makes Diagram \eqref{diag lKanformonoidality} commute for $F_i$. Furthermore, the functors $F_i$ are separately enriched cocontinuous in each variable. By the universal property of the category of  presheaves, there is a unique such dashed arrow in Diagram \eqref{diag lKanformonoidality}, see Kelly \cite[Theorem 4.51]{Kelly}. In particular, we obtain that there is a natural  isomorphism $F_2 \cong F_1$ that is compatible with the natural isomorphisms that make the diagram above commute. This isomorphism serves as the monoidal structure map  
\begin{equation*}\label{eq nat iso for strong monoidality}
    D(A) \wedge D(B) \cong D(A \otday B)
\end{equation*}
of $D$. This  map is commutative, associative and unital as desired.  
\end{proof}

\subsection{The localised presheaf category is nice} \label{sec Lpresheaf is nice}

Here, we prove that $L_S[\cat M_{\lambda,cof}^{\op},\cat V]$ is nice in the sense of Pavlov and Scholbach \cite[Definition 2.3.1]{pavlov2019symmetricopSsp}. This ensures that symmetric spectra in $L_S[\cat M_{\lambda,cof}^{\op},\cat V]$ is an admissible monoidal model category. A monoidal model category is \textbf{nice} if it is left proper, pretty small (Pavlov-Scholbach \cite[Definition 2.1]{pavlov2018htpythrysymmpowers}), $h$-monoidal (Batanin-Berger \cite[Definition 1.11]{batanin2017algoverpolynomial}), flat \cite[Definition 3.2.4]{pavlov2018htpythrysymmpowers} and tractable. We first show that $[\cat M_{\lambda,cof}^{\op},\cat V]$ is nice. 

A cofibrantly generated model category $\cat C$ is pretty small if a set of generating cofibrations for $\cat C$ have $\aleph_0$-presentable domains and codomains. The presheaf category $[\cat M_{\lambda,cof}^{\op},\cat V]$ is pretty small due to Lemma \ref{lemma presheaves is pretty small}.

Recall that every object is cofibrant in $\cat V$, see Notation \ref{notation axioms for V}. This implies that $\cat V$ is strongly $h$-monoidal, see Batanin and Berger  \cite[Lemma 1.12]{batanin2017algoverpolynomial}, i.e.\ $\cat V$ is $h$-monoidal (Batanin-Berger \cite[Definition 1.11]{batanin2017algoverpolynomial}) and the monoidal product preserves weak equivalences between all objects. Therefore by Theorem 4.1 of Batanin and Berger \cite{batanin2017algoverpolynomial}, $[\cat M_{\lambda,cof}^{\op},\cat V]$ is also strongly $h$-monoidal and left proper. 

The tractability of $[\cat M_{\lambda,cof}^{\op},\cat V]$ follows by \eqref{eq gencof for presheaves}. The following proposition completes the proof of our claim that $[\cat M_{\lambda,cof}^{\op},\cat V]$ is nice.

\begin{proposition}
The monoidal model category $[\cat M_{\lambda,cof}^{\op},\cat V]$ is flat in the sense of Pavlov-Scholbach \cite[Definition 3.2.4]{pavlov2018htpythrysymmpowers}. In other words, pushout product of a cofibration and a weak equivalence is a weak equivalence in $[\cat M_{\lambda,cof}^{\op},\cat V]$. 
\end{proposition}
\begin{proof}
Given a cofibration $y \co Y_1 \cofib Y_2$ and a weak equivalence $s \co S_1 \we S_2$ in $[\cat M_{\lambda,cof}^{\op},\cat V]$, we have the following diagram. 
\begin{equation*} 
 \begin{tikzcd}
 Y_1 \otday S_1 \ar[r,"\simeq"] \ar[d]& Y_1 \otday S_2 \ar[d] \ar[ddr,bend left = 20] & \\
 Y_2 \otday S_1 \ar[r]\ar[drr,"\simeq", bend right=10] &Y_2 \otday S_1 \coprod_{Y_1 \otday S_1} Y_1 \otday S_2  \ar[dr] & \  \\
 \ &\  & Y_2 \otday S_2  
 \end{tikzcd}
 \end{equation*}
The morphisms marked as weak equivalences are weak equivalences because $[\cat M_{\lambda,cof}^{\op},\cat V]$ is strongly $h$-monoidal and therefore the monoidal product on it preserves weak equivalences between all objects. This also implies that $y \wedge S_1$ is an $h$-cofibration. Pushouts along $h$-cofibrations preserve weak equivalences, therefore the bottom horizontal map is a weak equivalence. By 2 out of 3 property of weak equivalences, we deduce that the pushout product of $y$ and $s$ is also a weak equivalence.
\end{proof}
Let $\cat C$ be a left proper, pretty small, tractable and flat monoidal model category and let $C$ be a set of morphisms between cofibrant objects in $\cat C$. In this case, we say that the left Bousfield localization $L_C$ is a \textbf{monoidal left Bousfield localization} if $f \otimes A$ is a $C$-local equivalence for every cofibrant $A$ in $\cat C$ and morphism $f \in C$. This in particular guarantees that $L_C \cat C$ is a monoidal model category, see Gorchinskiy-Guletskii \cite[Lemma 31]{gorchinskiy2016symmetricpwrsinabstracthtpycat}. The property of being nice is preserved by monoidal left Bousfield localizations. Pretty smallness only depends on the cofibrations of the given category and left Bousfield localization preserve the properties of being left proper and tractable, see Barwick \cite[Proposition 4.12]{Barwick-loc}. Furthermore, flatness and $h$-monoidality are preserved due to Pavlov-Scholbach \cite[Proposition 6.4]{pavlov2018htpythrysymmpowers}. 

\begin{proposition}
The left Bousfield localization $L_S$ defined in  \eqref{eq localization maps} is a monoidal left Bousfield localization. Furthermore, $L_S[\cat M_{\lambda,cof}^{\op},\cat V]$ is an $\cat V$-model category. 
\end{proposition}
\begin{proof}
We start with the first statement. Let $f \in S$ and $A$ be a cofibrant object in $[\cat M_{\lambda,cof}^{\op},\cat V]$. Since $[\cat M_{\lambda,cof}^{\op},\cat V]$ is a left proper, pretty small, tractable and flat monoidal model category, it is sufficient to show that $f \otday A$ is a weak equivalence after localization. By definition, $S$ only contains cofibrations between cofibrant objects. Therefore $f \otday A$ is a map between cofibrant objects. 

Since the Quillen pair $D \dashv E$ is a Quillen equivalence after localization, a map between cofibrant objects in $[\cat M_{\lambda,cof}^{\op},\cat V]$ is an $S$-local equivalence if and only if its image under $D$ is a weak equivalence in $\cat M$. Therefore it is sufficient to show that $D(f \otday A)$ is a weak equivalence. 

By Proposition \ref{prop strongmonoidalDE}, $D$ is a monoidal functor. We have 
\[D(f\otday A) \cong D(f) \wedge D(A).\]
Since $f$ is an $S$-local equivalence between cofibrant objects, $D(f)$ is a weak equivalence between cofibrant objects. In a monoidal model category,  monoidal product with a cofibrant object is a left Quillen functor,  therefore it preserves weak equivalences between cofibrant objects. Furthermore $D(A)$ is cofibrant, therefore $D(f) \wedge D(A)$ is a weak equivalence. 

Now we prove the second statement. Since the cofibrations of $L_S[\cat M_{\lambda,cof}^{\op},\cat V]$ are the same as the cofibrations of $[\cat M_{\lambda,cof}^{\op},\cat V]$, we only need to prove the case of SM7 for a generating trivial cofibration $f\co A_1 \to A_2$ in $L_S[\cat M_{\lambda,cof}^{\op},\cat V]$ and a generating  cofibration $g \co B_1 \to B_2$ in $\cat V$. Indeed, we only need to show that the map $f \square g$, the pushout product of $f$ and $g$, is an $S$-local  equivalence. 

Since $L_S[\cat M_{\lambda,cof}^{\op},\cat V]$ and $\cat V$ are tractable, we assume that $A_1$, $A_2$, $B_1$ and $B_2$ are cofibrant objects. Due to SM7 in $[\cat M_{\lambda,cof}^{\op},\cat V]$, this guarantees that $f \square g$ is also a map between cofibrant objects. Therefore, it is sufficient to show that  $D(f\square g)$ is a weak equivalence. Since $D$ is a left adjoint functor that commutes with $\cat V$-tensor (see \eqref{eq def of D}), it preserves pushout products. In particular, 
\[D(f \square g) \cong D(f)\square g.\]
Since $D$ is a left Quillen functor, $D(f)$ is a trivial cofibration in $\cat M$ and therefore $D(f) \square g$ is a weak equivalence.

\end{proof}

\subsection{Admissible replacement}

We have shown that the given simplicial and combinatorial monoidal model category  $\cat M$ (with cofibrant unit)  is monoidally Quillen equivalent to the nice monoidal model category $L_S[\cat M_{\lambda,cof}^{\op},\cat V]$.  
Therefore, in order to prove Theorem \ref{thm admisslbe rplcmnt}, we need to show that $L_S[\cat M_{\lambda,cof}^{\op},\cat V]$ is monoidally Quillen equivalent to an admissible model category. For this, we use the main result of  Pavlov-Scholbach \cite{pavlov2019symmetricopSsp}.  

Pavlov and Scholbach work in a more general setting than Hovey. They consider modules over a commutative monoid in symmetric sequences. Let $\mathbb{I}$ denote the monoidal unit of $L_S[\cat M_{\lambda,cof}^{\op},\cat V]$. 
Let $E$ be the symmetric sequence given by $\I$ at each degree equipped with the trivial action. Let $Sp(L_S[\cat M_{\lambda,cof}^{\op},\cat V],E)$ denote the category of $E$-modules in the category of symmetric sequences in $L_S[\cat M_{\lambda,cof}^{\op},\cat V]$. Indeed, the category $Sp(L_S[\cat M_{\lambda,cof}^{\op},\cat V],E)$  is equivalent to the category of $I$-spaces in $L_S[\cat M_{\lambda,cof}^{\op},\cat V]$ (Sagave and Schlichtkrull \cite{sagave2012diagramspacessymspectra}), see Pavlov-Scholbach \cite[Proposition 3.2.2]{pavlov2019symmetricopSsp}.

\begin{theorem} \label{thm sp is admissble for every nice}
(Pavlov-Scholbach \cite[1.1]{pavlov2019symmetricopSsp}) For a nice monoidal model category $\cat C$, $Sp(\cat C,R)$ with the positive stable model structure is an admissible model category where $R$ denotes a commutative monoid in symmetric sequences in \cat C. 
\end{theorem}

In particular, $Sp(L_S[\cat M_{\lambda,cof}^{\op},\cat V],E)$ equipped with the positive stable model structure is admissible. The following proposition is due to Theorem 9.1 of Hovey \cite{hovey2001symspectrageneral}, see also Pavlov-Scholbach \cite[Example 3.3.2]{pavlov2019symmetricopSsp}.

\begin{remark}
The proposition below does not follow Hovey's notation. In Hovey's notation, one would  replace $E$ with the monoidal unit of $L_S[\cat M_{\lambda,cof}^{\op},\cat V]$.
\end{remark}

\begin{proposition} \label{prop spectral replacement}
The left Quillen functor Hovey \cite[Definition 7.3]{hovey2001symspectrageneral}
\[F_0 \co L_S[\cat M_{\lambda,cof}^{\op},\cat V] \to Sp(L_S[\cat M_{\lambda,cof}^{\op},\cat V],E)\] 
is the left adjoint of a  symmetric monoidal Quillen equivalence. Here, $Sp(L_S[\cat M_{\lambda,cof}^{\op},\cat V],E)$ is given the stable model structure.
\end{proposition} 

\begin{remark}
Hovey works in the setting of left proper cellular  model categories in \cite{hovey2001symspectrageneral}. The cellularity assumption is needed in order to make sure that certain left Bousfield localizations exists. Namely, in order to localize the projective model structure on symmetric spectra to obtain the stable model structure. In our case, $L_S[\cat M_{\lambda,cof}^{\op},\cat V]$ is indeed cellular but we do not need this. Since $L_S[\cat M_{\lambda,cof}^{\op},\cat V]$ is left proper and combinatorial, the relevant left Bousfield localizations are guaranteed to exist in $Sp(L_S[\cat M_{\lambda,cof}^{\op},\cat V])$ which is also left proper and combinatorial. 
\end{remark}

Finally, we prove Theorems \ref{thm admisslbe rplcmnt}, \ref{thm weak equivalences of operads} and \ref{thm infty category of algebras}.  The following theorem is a  generalization of Theorems \ref{thm admisslbe rplcmnt} and \ref{thm weak equivalences of operads} since the model category of simplicial sets satisfy the axioms for $\cat V$ given in Notation \ref{notation axioms for V}.  
\begin{theorem}\label{thm generalized admssible and w e of operads}
Let \cat V be as in Notation \ref{notation axioms for V}. Every combinatorial $\cat V$ symmetric monoidal model category is symmetric monoidally Quillen equivalent to an  admissible symmetric monoidal model category where a  weak equivalence of 
operads induce a Quillen equivalence between the model categories of the corresponding algebras.
\end{theorem}
\begin{proof}[Proof of Theorems \ref{thm generalized admssible and w e of operads} and \ref{thm infty category of algebras}]
We start with the proof of Theorem \ref{thm generalized admssible and w e of operads}. Let \cat N be a symmetric monoidal model category satisfying the hypothesis of the theorem. Theorem \ref{thm Muro cofibrant unit} states that \cat N is strong monoidally Quillen equivalent to another symmetric monoidal model category \cat M that also satisfies the hypothesis of the theorem and has a cofibrant monoidal unit. 

Propositions \ref{prop replacement with the presheaf} and \ref{prop strongmonoidalDE} state  that $\cat M$ is strong monoidally Quillen equivalent to the presheaf category $L_S[\cat M_{\lambda,cof}^{\op},\cat V]$. Furthermore, we prove in Section \ref{sec Lpresheaf is nice} that $L_S[\cat M_{\lambda,cof}^{\op},\cat V]$ is nice. 

Finally, $L_S[\cat M_{\lambda,cof}^{\op},\cat V]$  is strong monoidally Quillen equivalent to $Sp(L_S[\cat M_{\lambda,cof}^{\op},\cat V],E)$ with the stable model structure due to Proposition \ref{prop spectral replacement}.
With the stable model structure, $Sp(L_S[\cat M_{\lambda,cof}^{\op},\cat V],E)$ may not be admissible. However, this model structure is monoidally Quillen equivalent to the positive stable model structure through the identity functor, Pavlov-Scholbach \cite[Proposition 3.3.1]{pavlov2019symmetricopSsp}. As desired, the positive stable model structure on $Sp(L_S[\cat M_{\lambda,cof}^{\op},\cat V],E)$  is admissible due to Proposition \ref{thm sp is admissble for every nice}. This provides the admissible replacement of $\cat N$.

Furthermore, weak equivalence of operads induce a Quillen equivalence of the corresponding model categories of algebras in $Sp(L_S[\cat M_{\lambda,cof}^{\op},\cat V],E)$ due to Pavlov-Scholbach \cite[Theorem 4.6]{pavlov2019symmetricopSsp}. This finishes the proof of Theorem \ref{thm generalized admssible and w e of operads}.

For Theorem \ref{thm infty category of algebras}, we work in the setting where $\cat V $ is the model category of simplicial sets. In this situation, $Sp(L_S[\cat M_{\lambda,cof}^{\op},\cat \sS],E)$ satisfies the desired property due to Pavlov-Scholbach \cite[Theorem 4.9]{pavlov2019symmetricopSsp}.
\end{proof}

\section{Categories of modules} \label{sec categories of modules}
This section is devoted to the proof of Theorem \ref{thm morita classification}. In other words, we show that every stable, combinatorial and simplicial monoidal model category whose monoidal unit is a compact generator is  monoidally Quillen equivalent to modules over a commutative ring spectrum. Furthermore, we show that this commutative ring spectrum is unique.

\subsection{Replacement with a spectral model category}

Here, we show that every stable, combinatorial and simplicial symmetric monoidal model category \cat N is monoidally Quillen equivalent to a spectral  one. 

As before, one replaces $\cat N$ with another stable, combinatorial and simplicial symmetric monoidal model category $\cat M$ whose monoidal unit is cofibrant, see Theorem \ref{thm Muro cofibrant unit}. Since $\cat N$ is pointed and since   $\cat M$ is equivalent to \cat N as a category, $\cat M$ is also pointed. 

 Since \cat M is a pointed simplicial model category, it is also an $\sS_*$-model category in a natural way where $\sS_*$ denotes the category of pointed simplicial sets, Hovey \cite[Definition 4.2.19]{Hovey}.
 As before, $\cat M$ is strong monoidally Quillen equivalent to $L_S[\cat M_{\lambda,cof}^{\op},\sS_*]$ where $\lambda$ is as in Section \ref{sec replacement with presheaf category} and $S$ is as in Proposition \ref{prop replacement with the presheaf}. This follows by Propositions \ref{prop replacement with the presheaf} and \ref{prop strongmonoidalDE}.
 
 \begin{notation}\label{notation positive stable model structure}
For a given nice $\sS_*$ symmetric monoidal model category $(\cat C,\otimes_{\cat C},\I_{\cat C})$, let $K$ denote the symmetric sequence in \cat C given by $K_n = (S^1 \otimes \I_{\cat C})^{\otimes_{\cat C}n}$ at degree $n$. We denote the  model category of $K$-modules equipped with the stable model structure by $Sp(\cat C)$.  Let $A$ be a monoid in $Sp(\cat C)$. We denote the  category of $A$-modules by $A\textup{-mod}$ and unless otherwise stated, we assume that this category is equipped with the stable model structure. The  category of $A$-modules  equipped with the \textit{positive} stable model structure is denoted by $A\textup{-mod}^+$. Similarly, we write $Sp(\cat C)^+$ when we are using the positive stable model structure on $Sp(\cat 
 C)$.
 \end{notation}
 \begin{notation}\label{notation for D}
 For the rest of this work, let $(\cat D, \otimes_{\cat D}, \I_{\cat D})$ denote a stable,  combinatorial and nice  $\sS_*$ symmetric monoidal model category. Furthermore, assume that  $\I_{\cat D}$ is cofibrant and assume that $\I_{\cat D}$ is a compact generator of the homotopy category of \cat D. We denote the monoidal product and the monoidal unit of  $Sp(\cat D)$ as in $(Sp(\cat D),\wedge,\I)$.
 \end{notation}

The following is a consequence of  Hovey \cite[Theorem 9.1]{hovey2001symspectrageneral}. 
\begin{proposition}\label{prop spectra in a stable model category}
Let $\cat D$ be as in Notation \ref{notation for D}.  In this situation, $Sp(\cat D)$ is strong symmetric monoidally Quillen equivalent to $\cat D$. Furthermore, the monoidal unit of $Sp(\cat D)$ is a compact generator of the homotopy category of $Sp(\cat D)$.
\end{proposition}

 Ultimately, we are interested in the case  $\cat D = L_S[\cat M_{\lambda,cof}^{\op},\sS_*]$. We have the following.
 
\begin{proposition}\label{prop stable spectral  replacement}
Let $\cat N$ be a symmetric monoidal model category satisfying the hypothesis of Theorem \ref{thm morita classification}. Also, let \cat M, $\lambda$ and $S$ be as above. The presheaf category $L_S[\cat M_{\lambda,cof}^{\op},\sS_*]$ satisfies the assumptions for $\cat D$ given in Notation \ref{notation for D}. Furthermore, \cat N is strong symmetric  monoidally Quillen equivalent to  $Sp(L_S[\cat M_{\lambda,cof}^{\op},\sS_*])$. 
\end{proposition}

  In order to prove Theorem \ref{thm morita classification}, we need to show that $Sp(\cat D)$ is monoidally Quillen equivalent to the model category of $R$-modules for some commutative ring spectrum $R$. For this, we need to consider a spectral enrichment of $Sp(\cat D)$.
 \begin{construction}\label{construction spectral enrichment in D}
 Since \cat D is an $\sS_*$ monoidal model category, there is a left Quillen functor 
 $T \co  \sS_* \to \cat D$.
 Following the discussion after Definition 7.2 of Hovey \cite{hovey2001symspectrageneral} and  Hovey  \cite[Theorem 8.11]{hovey2001symspectrageneral}, one obtains that $Sp(\cat D)$ is a spectral symmetric monoidal model category. In other words, there is a canonical monoidal left  Quillen  functor
 \[\tilde{T} \co   Sp(\sS_*) \to Sp(\cat D).\]
 We denote the right adjoint of $\tilde{T}$ by $\tilde{U}$.
 
 \end{construction}
 
 \begin{proposition}\label{prop tilde T between positive}
 The adjoint pair $\tilde{T} \dashv \tilde{U}$ in Construction \ref{construction spectral enrichment in D} is also a Quillen pair between $Sp(\sS_*)^+$ and $Sp(\cat D)^+$.
 \end{proposition}
 \begin{proof}
We start by describing the functor  $\tilde{T}$ in detail. Given a category \cat C, let $\Sigma \cat C$ denote the category of symmetric sequences in \cat C.
There is a  symmetric sequence $S$ in $\sS_*$ is given by $(S^1)^{\wedge k}$ in degree $k$. Furthermore, recall that $K$ is the symmetric sequence in  $\cat D$ given by $K_k = (S^1)^{\wedge k} \otimes \II_{\cat D}$. With this notation, the categories of $S$-modules and $K$-modules are $Sp(\sS_*)$ and  $Sp(\cat D)$ respectively. 

Recall that there is a monoidal left Quillen functor 
\[T \co \sS_* \to \cat D.\]
Let $U$ denote the right adjoint of $T$. Applying $T$ levelwise makes $\Sigma \cat D$ into a $\Sigma \sS_*$ monoidal model category, see Hovey  \cite[Section 7]{hovey2001symspectrageneral}. In other words, the functor $\Sigma \sS_* \to \Sigma \cat D$ given by $T$ at each level is  monoidal. Furthermore, the right adjoint to this functor is also given by $U$ at each level as explained in Hovey \cite[Section 7]{hovey2001symspectrageneral}.
The levelwise application of $T$ to $S$ gives $K$ and therefore there is an adjoint pair  between $S$-modules and $K$-modules. This is indeed the adjoint pair $\tilde{T} \dashv \tilde{U}$. This shows that $\tilde{U}$ is given by levelwise application of $U$ on the underlying symmetric sequences.

Since localizations do not change trivial fibrations, the trivial fibrations in the positive stable model structure of $Sp(\sS_*)$ ($Sp(\cat D)$) are given by the maps that are trivial fibrations at each level of $\Sigma \sS_*$ ($\Sigma \cat D$) except possibly in degree zero, see Pavlov and Scholbach \cite[Notation 2.3.5]{pavlov2019symmetricopSsp}. Since $\tilde{U}$ applies the right Quillen  functor $U$  levelwise, it preserves the trivial fibrations of the positive stable model structure. This also shows that $\tilde{T}$ preserves the cofibrations of the positive stable model structure. 

Let $f$ be an acyclic cofibration in the positive stable model structure on $Sp(\sS_*)$. We already showed that $\tilde{T}(f)$ is a cofibration in the positive stable model structure. Therefore it is sufficient to show that $\tilde{T}(f)$ is a weak equivalence in the positive stable model structure.  Since there are more trivial fibrations in the positive stable model structure than the stable model structure, a cofibration in the positive stable model structure is also a cofibration in the stable model structure. Furthermore, the weak equivalences of the positive stable model structure and the stable model structure agree, see Pavlov-Scholbach  \cite[Proposition 3.3.1]{pavlov2019symmetricopSsp}. In particular, $f$ is also an acyclic cofibration in the stable model structure. Therefore $\tilde{T}(f)$ is an acyclic cofibration in the stable model structure. This shows that $\tilde{T}(f)$  is a weak equivalence in the positive stable model structure as desired. 
 \end{proof}
 
\subsection{Proof of Theorem \ref{thm morita classification}}

To prove Theorem \ref{thm morita classification}, we need to show that $Sp(\cat D)$ is monoidally Quillen equivalent to the model category $R$-modules for some commutative ring spectrum $R$. Here,  $\cat D$ is as in Notation \ref{notation for D}. The definition of the symmetric monoidal model category $(Sp(\cat D), \wedge, \I)$ is given in Notation \ref{notation positive stable model structure}. Our conventions on stable and positive stable model structures are also given in Notation \ref{notation positive stable model structure}.

 The commutative ring spectrum $R$ should be considered as the derived endomorphism ring spectrum of the monoidal unit of $Sp(\cat D)$. However, in order to ensure that the endomorphism ring spectrum is a commutative ring spectrum, one needs a fibrant model of $\I$ as a commutative monoid. For this, we take a fibrant replacement
 \begin{equation}\label{eq fibrant replacement of I}
      \varphi \co \I \we f\I
 \end{equation}
of $\I$ in the model category of commutative monoids in $Sp(\cat D)^+$. Since the weak equivalences of the positive stable model structure agree with those of the stable model structure Pavlov-Scholbach \cite[Proposition 3.3.1]{pavlov2019symmetricopSsp}, $\varphi$ is also a weak equivalence in $Sp(\cat D)$.

Since $Sp(\cat D)$ is flat (Pavlov-Scholbach \cite[Proposition 3.4.2]{pavlov2019symmetricopSsp}), Quillen invariance holds for $Sp(\cat D)$. This means that the Quillen adjunction between the model categories of $\II$-modules and $f\II$-modules induced by $\varphi$ is a Quillen equivalence, Schwede and Shipley \cite[Theorem 4.3]{schwede2000algebrasandmodules}. Similarly, this adjunction is also a Quillen equivalence between $\I\textup{-mod}^+$ and $f\I\textup{-mod}^+$. Since $\II$ denotes the monoidal unit of $Sp(\cat D)$, $\II\textup{-mod}$ is simply another name for $Sp(\cat D)$. Because $f\II$ is a commutative monoid, $f\II\textup{-mod}$ is also a closed symmetric monoidal model category. The left adjoint of the Quillen equivalence between $\II\textup{-mod}$ and $f\II\textup{-mod}$ is a  monoidal Quillen functor, see the discussion after Definition 4.1.14 in Hovey \cite{Hovey}. This functor is given by 
\[f\II \wedge -\co Sp(\cat D) \to f\I\textup{-mod}.\]
 The  monoidality of this functor follows from the natural isomorphism
\[f\II \wedge (X\wedge Y)  \cong (f\II \wedge X) \wedge_{f\II} (f\II \wedge Y)\]
that holds for every $X,Y \in Sp(\cat D)$. Here,  $\wedge_{f\II}$ denotes the monoidal product in $f\II$-modules. 

This shows that $Sp(\cat D)$ is  monoidally Quillen equivalent to the model category of $f\II$-modules. Therefore, it is sufficient to show that the model category of $f\II$-modules is  monoidally Quillen equivalent to modules over a commutative ring spectrum. We obtain the following. 
\begin{proposition}\label{prop fimodules}
Let $f\I$ be the fibrant replacement of $\I$ given in \eqref{eq fibrant replacement of I}. The functor $f\I \wedge - $ is the left adjoint of a monoidal Quillen equivalence between ($Sp(\cat D)^+$) $Sp(\cat D)$ and  ($f\II\textup{-mod}^+$) $f\II\textup{-mod}$.  
\end{proposition}
\begin{construction}\label{construction L and G}
In Construction \ref{construction spectral enrichment in D}, we show that  there is a monoidal left Quillen functor $\tilde{T} \co Sp(\sS_*) \to Sp(\cat D).$
Composing this with the  monoidal left Quillen functor $f\II \wedge - $, we obtain a  monoidal left Quillen functor
\[F := (f \I \wedge -) \circ \tilde{T} \co Sp(\sS_*) \to f\II\text{-mod}.\]
 Let $R$ denote the right adjoint of $F$. 

Since $F$ is monoidal, $R$ is lax monoidal and $R(f\I)$ is a commutative ring spectrum. Again because it is monoidal, $F$ induces a functor
\[F' \co R(f\I)\textup{-mod} \to FR(f\I)\textup{-mod}.\]
Here, $F'$ is given by $F$ and the right adjoint $R'$ of $F'$ is again given by  $R$. Since $F \dashv R$ is a strong monoidal Quillen pair,  the counit map of this adjunction  provides a map $FR(f\II) \to f\II$ of commutative monoids. Through this map, we obtain a left Quillen functor
\[f\II \wedge_{FG(f\II)} - \co  FR(f\I)\textup{-mod} \to  f\II\textup{-mod}.\]

Finally, let $L= (f\II \wedge_{FG(f\II)} -) \circ F'$. We obtain a  Quillen pair $L \dashv G$  as shown below.
\begin{equation*}
\begin{tikzcd}
    L   \co R(f\I)\textup{-mod} \ar[r,shift left]& \ar[l,shift left] f\I\textup{-mod} \co G
\end{tikzcd}
\end{equation*}
Going through the definition of these functors above, one observes that $G$ is given by $R$. 
\end{construction}
To prove Theorem \ref{thm morita classification}, we need to show that $L \dashv G$ is a monoidal Quillen equivalence.
\begin{remark}
In the construction above, although we do not make it explicit, we make use of the spectral enrichment of  $f\I\textup{-mod}$  as in the proof of   Schwede-Shipley \cite[Theorem 3.9.3]{schwede2003stablearemodules}. The mapping spectrum $\hom(-,-)$ in the category of $f\I$-modules is given by $R(\Hom(-,-))$ where $\Hom(-,-)$ denotes the internal hom in $f\I\textup{-mod}$.  The functor $G$ above is given by 
\[G(-) = R(-) = R(\Hom(f\I,-)) = \hom(f\I,-).\]
In particular, one observes that the adjoint pair $L \dashv G$ above agrees with the one constructed in the proof of Schwede-Shipley \cite[Theorem 3.9.3]{schwede2003stablearemodules}.
\end{remark}

\begin{proposition}
The Quillen pair $L \dashv G$ in Construction \ref{construction L and G} is also a Quillen pair between $R(f\I)\textup{-mod}^+$ and $f\I\textup{-mod}^+$.
\end{proposition}
\begin{proof}
It is sufficient to show that $G$ preserves positive (trivial) fibrations. In the model category of modules over a given monoid in a nice monoidal model category, (trivial) fibrations are those of the underlying model category. Since the $G$ is given by $R$, it is sufficient to show that $R$ preserves positive (trivial) fibrations.  For this, it is sufficient to show that $F$ is a left Quillen functor between $Sp(\sS_*)^+$ and $f\I\textup{-mod}^+$.  This follows by Propositions \ref{prop tilde T between positive} and \ref{prop fimodules}.
\end{proof}

\begin{proposition}\label{prop L G is strong monoidal}
The left Quillen functor $L$ is monoidal. In other words, $L \dashv G$ is a strong monoidal Quillen pair.
\end{proposition}
\begin{proof}
Since $F$ is a composite of  monoidal functors, it is a  monoidal functor. Since $F$ preserves coequalizers and monoidal products, one observes that it also carries monoidal products of $R(f\I)$-modules to monoidal products of $FR(f\I)$-modules. In other words, $F'$ is also  monoidal. Finally $L$ is monoidal as it is given by a composite of monoidal functors.
\end{proof}

The proof of Schwede-Shipley \cite[Theorem 3.9.3]{schwede2003stablearemodules} goes through to show that $L \dashv G$ is a Quillen equivalence between the respective stable model structures. The only difference in our situation and that of Schwede-Shipley \cite[Theorem 3.9.3]{schwede2003stablearemodules} is that our compact generator $f\I$ may not be  bifibrant in $f\I\textup{-mod}$. It is fibrant in $f\I\textup{-mod}^+$ but it may not be fibrant in $f\I\textup{-mod}$. Similarly, $f\I$ is cofibrant in $f\I\textup{-mod}$ but it may not be cofibrant in $f\I\textup{-mod}^+$. 
\begin{theorem}\label{thm schwede shipley morita} (Schwede-Shipley \cite[Theorem 3.9.3]{schwede2003stablearemodules})
As in Equation \eqref{eq fibrant replacement of I}, let $\I \we f\I$ be a fibrant replacement of the monoidal unit $\I$ of $Sp(\cat D)$ in the model category of commutative monoids in $Sp(\cat D)^+$. The Quillen pair 
\[L \dashv G\]
given in Construction \ref{construction L and G} is a Quillen equivalence between  $R(f\I)\textup{-mod}$ and $f\I\textup{-mod}$.
\end{theorem}
\begin{proof}
We first show that the derived unit map
\[R(f\I) \to Gf'Lc'(R(f\I))\]
 of the adjunction $L \dashv G$ (evaluated at $R(f\I)$) is a weak equivalence. Here, $f'$ denotes a fibrant replacement functor in $f\I\textup{-mod}$ and $c'$ denotes a cofibrant replacement functor in $R(f\I)\textup{-mod}$. As $R(f\I)$ is cofibrant in $R(f\I)\textup{-mod}$, the cofibrant replacement functor may be omitted. Also, $L(R(f\I)) = f\I$ as $L$ is monoidal. The monoidal unit $f\I$ is only assumed to be  fibrant in the positive stable model structure and therefore it may not be fibrant in the non-positive stable model structure. However, $G$ is also a right Quillen functor between the respective positive stable model structures and the weak equivalences of the positive and the non-positive stable model structures agree Pavlov-Scholbach  \cite[Proposition 3.3.1]{pavlov2019symmetricopSsp}. Due to this, the fibrant replacement $f'$ above  may also be omitted. Since $G$ is  given by $R$ and since $L$ is monoidal, one obtains that \[GL(R(f\I)) = G(f\I) = R(f\I)\] 
 as desired. Similarly, one shows that the derived counit map evaluated at $f\I$
 \[Lc'Gf'(f\I) \to f\I\]
 is also a weak equivalence where $f'$ and $c'$ denote the relevant fibrant and cofibrant replacement functors respectively. 
 
 The  rest of the proof follows as in Schwede-Shipley \cite[Theorem 3.9.3]{schwede2003stablearemodules}. For the sake of completeness, we provide a sketch of the proof of Schwede-Shipley  \cite[Theorem 3.9.3]{schwede2003stablearemodules} in our case, i.e. in the case of a single compact generator. For this, it is sufficient to show that the induced adjoint pair at the level of homotopy categories is indeed an equivalence of categories. The unit and the counit maps of the derived adjunction are isomorphisms when evaluated at the respective monoidal units as we show above. 
 
It follows from Proposition \ref{prop spectra in a stable model category} and our constructions that  $f\II$ is a compact generator of the homotopy category of $f\II\textup{-mod}$.

The derived functor of $L$ is a left adjoint and therefore it preserves coproducts. Using the compactness of $f\II$, one can show as in the proof of Schwede-Shipley  \cite[Theorem 3.9.3]{schwede2003stablearemodules} that the derived functor of $G$ also preserves coproducts. Furthermore, both derived functors preserve shifts and triangles. Therefore, the unit and the counit maps of the derived adjunction are isomorphisms on the categories generated by the monoidal units under coproducts, triangles and shifts. Since both monoidal units are generators, this shows that the counit and the unit maps of the derived adjunction are isomorphisms on all objects. 
\end{proof}

The following provides the uniqueness part of Theorem \ref{thm morita classification}. This is a consequence of Lurie's results. 
\begin{proposition}\label{prop unqiueness part of theorem morita}
Let $\cat M$ be a symmetric monoidal model category as in Theorem \ref{thm morita classification} with a cofibrant monoidal unit. Furthermore, assume that  $\cat M$ is (strong or weak) symmetric monoidally Quillen equivalent to $R_i$-modules via a zig-zag of symmetric monoidal Quillen equivalences where each symmetric monoidal model category involved in the zig-zag has a cofibrant monoidal unit. Here, $R_i$ denotes  a  commutative ring spectrum for $i=1$ and $i=2$. In this situation, $R_1$ and $R_2$ are weakly equivalent as commutative ring spectra.
\end{proposition}

\begin{proof}
The hypothesis implies that the symmetric monoidal $\infty$-categories corresponding to the model categories of  $R_1$-modules and $R_2$-modules are equivalent, see P\'eroux \cite[Theorem 2.13]{peroux2020coalgebras}. It follows from Proposition 7.1.2.7 of Lurie \cite{lurie2012higher} that $R_1$ and $R_2$ are weakly equivalent as commutative ring spectra.
\end{proof}

We are ready to prove Theorem \ref{thm morita classification}.

\begin{proof}[Proof of Theorem \ref{thm morita classification}]
Let \cat N be a  stable, combinatorial and simplicial monoidal model category whose monoidal unit is a compact generator, i.e.\ let \cat N satisfy the hypothesis of Theorem \ref{thm morita classification}. Due to Proposition \ref{prop stable spectral  replacement}, \cat N is monoidally Quillen equivalent to $Sp(\cat D)$ for some monoidal model category $\cat D$ as in Notation \ref{notation for D}. Let $\I \we f\I$ be a fibrant replacement of the monoidal unit $\I$ of $Sp(\cat D)$ in the model category of commutative monoids in $Sp(\cat D)^+$. Due to Proposition \ref{prop fimodules}, $Sp(\cat D)$ is monoidally Quillen equivalent to $f\I\textup{-mod}$. 
Construction \ref{construction L and G} provides a Quillen adjunction $L \dashv G$ between the model categories of $f\I$-modules and $R(f\I)$-modules where $R(f\I)$ is a commutative ring spectrum. This Quillen adjunction is monoidal due to Proposition \ref{prop L G is strong monoidal} and it is indeed a Quillen equivalence due to Theorem \ref{thm schwede shipley morita}. This proves that the given \cat N is monoidally Quillen equivalent to the model category of modules over $R(f\I)$ (with the stable model structure) as desired.

The uniqueness part of Theorem \ref{thm morita classification} follows by Proposition \ref{prop unqiueness part of theorem morita}. 
\end{proof}

\end{document}